\documentclass[12pt]{siamart190516}

\usepackage[T1]{fontenc}
\usepackage{multirow}
\usepackage{amsmath}
\usepackage{amssymb}
\usepackage{textcomp}
\usepackage{tikz,pgfplots}
\usepackage{pgf,comment,color,url}
\usepackage{algorithmic,algorithm}

\newtheorem{remark}[theorem]{Remark}

\newcommand{\wh}{\widehat}
\newcommand{\bm}{\boldsymbol}
\newcommand{\E}{{\mathrm{E}}}
\newcommand{\diag}{{\mathrm{diag}}}

\newcommand{\ones}{\bm{1}}
\newcommand{\smin}{S_{\min}}
\newcommand{\wmin}{W_{\min}}
\newcommand{\wtb}{\widetilde{B}}

\newcommand{\whe}{\wh{E}}

\newcommand{\1}[1]{\indic[#1]} 
\newcommand{\indic}{\mathbb I}
\newcommand{\real}{\mathbb R}

\newcommand{\ud}{\mathrm{d}}   
\renewcommand{\qed}{\hfill $\square$}
\newcommand{\calg}{\mathcal{G}}         %
\newcommand{\levy}{L\'evy }                    
\newcommand{\whP}{\wh P}    
\newcommand{\tr}{^{\mbox{\tiny T}}}

\begin{document}

\title{Numerical solution of a matrix integral equation arising in Markov Modulated L\'evy processes\thanks{Research partially supported by INdAM-GNCS}}

\author{Dario Bini\footnote{Dipartimento di Matematica, Universit\`{a}
    di Pisa, Largo Bruno Pontecorvo, 5, 56127 Pisa,  Italy} \and
Guy Latouche\footnote{
Universit\'e libre de Bruxelles,
D\'epartement d'informatique, CP212, 1050 Bruxelles, Belgium} \and
Beatrice Meini$^{*}$}

\maketitle
\begin{abstract}
Markov-modulated \levy processes lead to matrix integral equations of the kind 
$  A_0 + A_1X+A_2 X^2+A_3(X)=0$
where
$A_0$, $A_1$, $A_2$ are given matrix coefficients, while $A_3(X)$ is a nonlinear function, expressed in terms of integrals involving the exponential of the matrix $X$ itself. 
In this paper we propose some numerical methods for the solution of this class of matrix equations,  perform a theoretical convergence analysis and show the effectiveness of the new methods by means of a wide numerical experimentation.
\end{abstract}



\section{Introduction}\label{sec:intro}
\subsection{The problem}
A Markov-modulated \levy process $\{(X_t, J_t)\}$ has two components:
$X_t$ takes real values and is called the level, $J_t$ is called the
phase, it is a continuous-time Markov process with irreducible generator
$Q$ and state space $\{ 1,\ldots,n\}$.  Over intervals of time where
$J(t)$ remains constant, and equal to $i$ say, the level behaves as
the superposition of a Brownian motion and Poisson processes of jumps,
with parameters dependent on $i$; when the phases switches from $i$ to
$j \not= i$, $X(t)$ jumps by a random quantity $Y_{ij}$.  Applications
are found in mathematical finance (Jobert and Rogers~\cite{jr06},
Ballotta and Bonfiglioli~\cite{bb16}, Deelstra and Simon~\cite{ds17}),
risk theory (Lu and Tsai~\cite{lc07}, Li and
Ren~\cite{lr13}) and queueing theory (Prabhu~\cite{prabh98}).

Such a process is characterised by the
matrix-valued function $F(s)$ such that $F_{ij} (s)  =\E[ e^{s X_t} \1{J_t = j} | X_0=0, J_0=i]$,
for $t \geq 0$, $1 \leq i,j \leq n$ and $s$ such that the expectation
is finite, where $\1{\cdot}$ is the indicator function; $F(s)$ is given by
\begin{equation}
   \label{e:fs}
F(s) = \diag(\psi_1(s), \ldots, \psi_n(s)) + \int_\real e^{s x}(Q \circ U(\ud x)) 
\end{equation}
where $\psi_i(s)$ is the Laplace exponent of the \levy process
corresponding to phase $i$, $U_{ij}(\cdot)$, $i \not= j$, is a probability distribution
  function on $\real$ and $U_{ii}(\cdot)$ is the step function:
  $U_{ii}(x) = 0$ for $x < 0$, $U_{ii}(x) = 1$ for $x \geq 0$ (see
  d'Auria {\it et al.}~\cite{dikm10}).   We
  define $A\circ B$ as the Hadamard product of $A$ and $B$.

Jumps during an interval of sojourn in a given phase are parametrised
by the so-called \levy measure (Bertoin~\cite[Page 3]{bertoin96}) which
we assume to be finite for all phases.
In physical terms, this means that jumps occur at epochs of Poisson
processes, instead of a denumerably infinite superposition of Poisson
processes.  This assumption is not very restrictive when one considers
practical applications, and it will make it easier to implement
numerical algorithms. 

Furthermore, we consider that jumps take positive values only.  In
that case, a matrix denoted as $G$ plays a fundamental role in the
analysis of the Markov-modulated \levy process, it is such that
$(e^{Gx})_{ij}$ is the probability that the phase is $j$ upon the first
visit to level $-x$, given that the process starts from level
0, in phase $i$ at time 0.  

Few authors have considered the question of computing $G$.  
Breuer~\cite[Theorem 2]{breue08} suggests a functional iteration
  procedure, but there is no indication that the author has
  implemented it, and he does not perform a convergence analysis.
 D'Auria {\it et al.}~\cite{dikm10} conduct a spectral analysis
  of the matrix function $F(s)$ and obtain a characterisation of the
  Jordan normal form of~$G$.  This, in principle, may serve as the
  basis for a numerical procedure, but they do not perform it. Moreover, we point out that
  a numerical procedure based on the computation of the Jordan normal
  form of a matrix is sensitive to numerical error and is deprecated
  for that reason.
Simon~\cite[Proposition 4.8]{matth17} defines a different
  functional iteration procedure, there is no implementation, nor
  convergence analysis.
One should add that these authors consider processes with negative
jumps only, but our results are easily adapted to that case.

The $n\times n$ matrix $G$ solves an integral matrix equation of the kind
\begin{equation}\label{eq0}
  A_0 + A_1G+A_2 G^2+A_3(G)=0
\end{equation}
where
$A_0$, $A_1$, $A_2$ are given $n\times n$ matrix coefficients, while $A_3(G)$ is a nonlinear matrix function, expressed in terms of integrals involving the exponential of the matrix $G$. 

\subsection{The results}
Our purpose in the present paper is to define and analyse
numerical procedures to compute $G$ by solving the matrix equation \eqref{eq0}. To our knowledge, a numerical analysis of matrix equations of the kind \eqref{eq0} has not been performed in the literature, while there is  wide literature on matrix equations of the kind $\sum_{i=0}^\infty A_i G^i=0$. We refer to the books \cite{blm:book} and \cite{lr99} for a survey on matrix equations arising in structured Markov chains, to 
\cite{higham-kim-2000} and \cite{higham-kim-2001} for 
the solution of general quadratic matrix equations, to  \cite{meng} for an analysis of the conditioning, and to 
\cite{blm20},
\cite{ren-can-li},
\cite{6cinesi},
\cite{guo-1999}, \cite{guo-2003}, 
\cite{benny-2012}, 
\cite{rhee-2010}, 
\cite{seokim}, and \cite{seo2018} 
 for strategies to improve the accuracy and the convergence.

The numerical methods that we propose, inspired by the recent paper~\cite{blm20}, consist in solving a sequence of
quadratic matrix equations, where the matrix coefficients defining the
matrix equation depend on the current approximation to the solution
$G$.  The numerical methods differ in the way the quadratic matrix
equations are generated.  For instance, a natural approach consists in
generating the sequence $\{ G_k\}_k$, where $G_{k+1}$ solves the
quadratic matrix equation $A_0+A_1 X+ A_2X^2+A_3(G_k)=0$, for
$k=0,1,\ldots$.  We elaborate this idea and  propose different
strategies for defining the equation and for solving it. 
A preliminary step consists in applying a change of variable, 
so that the solution of the new equation is stochastic.
In the first method   we obtain
a suitable unilateral quadratic matrix equation,
depending on a positive parameter $\tau$, with nonnegativity properties
of the matrix coefficients.  
In the second method, we obtain a special nonsymmetric algebraic Riccati equation.  A
theoretical analysis of the conditions of convergence and of the
convergence rate is carried out, together with the determination of the optimal value of $\tau$ for the first method. Numerical tests are performed together with
comparisons with the algorithms proposed by Breuer \cite{breue08} and
Simon \cite{matth17}. The numerical results confirm the theoretical analysis and show the effectiveness of our approach with respect to \cite{breue08} and \cite{matth17}.

The paper is organized al follows. In Section \ref{sec:not} we recall some properties on nonnegative matrices and introduce some notation. In Section \ref{sec:int} we describe the matrix integral equation and recall some properties of the solution~$G$.
In Sections \ref{sec:ex} and \ref{s:riccati} we introduce the changes
of variable and we rewrite the matrix integral equation in terms of a
unilateral quadratic matrix equation and of a nonsymmetric algebraic
Riccati equation, respectively, the matrix coefficients of which depend on $G$ itself. Section \ref{sec:com} is devoted to the design of three numerical methods for computing $G$, based on the different reductions to quadratic and Riccati matrix equations.  Finally, numerical results are presented and discussed in Section \ref{sec:num}.

\section{Notation and preliminaries}\label{sec:not}
A real matrix $A=(a_{i,j})_{i,j}$ is called nonnegative (positive), and we write $A\ge 0$ ($A>0$), if
$a_{i,j}\ge 0$ ($a_{i,j}>0$) for every $i,j$. We write $A\ge B$ ($A>B$) if $A-B\ge 0$ ($A-B>0$).  The
spectral radius of a matrix $A$ is denoted by $\rho(A)$.
A real matrix $A$ is called an M-matrix if it  can be written in the form $A=\sigma I-B$ where $B\ge 0$ and $\rho(B)\le \sigma$.
A real matrix $A$ is called a Z-matrix if $a_{ij}\le 0$ for $i\ne j$.
The following result gives sufficient conditions under which a Z-matrix is an M-matrix (see \cite[Chapter 6, Lemma 4.1, Theorem 2.3 and Exercise 5.1c]{bp:book}):

\begin{theorem}\label{thm:mm}
Assume that $A$ is a Z-matrix. If there exists a vector
$u>0$ such that $Au\ge 0$ then $A$ is an M-matrix; if in addition $Au>0$, or if $A$ is irreducible and $Au\ne 0$, then $A$ is a nonsingular
M-matrix. If $A$ is an M-matrix and $B\ge A$ is a Z-matrix, then $B$ is an M-matrix.
Finally,  if $A$ is a nonsingular M-matrix then $A^{-1}\ge 0$.
\qed
\end{theorem}

Given a real matrix $A$, the splitting $A=M-N$ is said to be a regular splitting if $\det M\ne0$, $M^{-1}\ge 0$ and $N\ge 0$.  The following result provides spectral properties of $M^{-1}N$ (see \cite[Theorems 3.13-3.15]{varga})

\begin{theorem}\label{thm:rs}
Let $A$ be a real matrix such that $A^{-1}\ge 0$. If $A=M-N$ is a regular splitting then $\rho(M^{-1}N)<1$. If
$A=M_1-N_1$ and $A=M_2-N_2$ are two regular splittings such that $N_1\le N_2$, then $\rho(M_1^{-1}N_1)\le \rho(M_2^{-1}N_2)$;
if, in addition, $A^{-1}>0$ and $N_1\ne N_2$ then $\rho(M_1^{-1}N_1)<
\rho(M_2^{-1}N_2)$.
\qed
\end{theorem}

A generator $A=(a_{i,j})_{i,j}$ is a square real matrix such that
$a_{i,j}\ge 0$ for $i\ne j$ and $A\ones=0$, where $\ones $ is the
vector with all entries equal to 1.  In particular, $-A$ is a
singular M-matrix.  The matrix $A$ is a  subgenerator if $a_{i,j}\ge 0$ for $i\ne j$ and $A\ones \le 0$.

For any square matrices $A$ and $E$, and for any $t\ge 0$, the
following identity holds \cite[Section 10.2]{higham:book}
\begin{equation}\label{eq:expdiff}
e^{(A+E)t}-e^{At}=\int_0^t e^{A(t-s)}E e^{(A+E)s}ds.
\end{equation}

Given a vector $y=(y_i)_{i=1,\ldots,n}$, we denote by $\Delta_y$ the $n\times n$ diagonal matrix with diagonal entries $y_1,\ldots,y_n$.
Given two matrices $A=(a_{i,j})_{ij}$ and $B=(b_{i,j})_{ij}$ having
the same dimension,  the expression $C=A\circ B$ denotes the Hadamard
product, where the entries of $C$ are $c_{i,j}=a_{i,j}b_{i,j}$.

\section{The matrix integral equation}\label{sec:int}

In this paper, the expression (\ref{e:fs}) for the 
matrix function $F(s)$ is more precisely the following:
\[
F(s)  
 = \frac12 \Delta_{\sigma^2} s^2 + \Delta_a s + \int_0^\infty \Delta_\nu(x)\left( 
e^{s x}-I\right) \ud x 
+ Q\circ U(0)+\int_0^\infty (Q\circ \mu(x))e^{sx}  \ud x,
\]
where $a=(a_i)$, $\sigma^2=(\sigma^2_i)$ and
\begin{itemize}
\item
 $a_i$, $\sigma_i\in\mathbb R$  define the Brownian motion
 while $J_t$ is in phase $i$;
\item
we assume that the \levy measures have a continuous density 
$\nu_i(\cdot)$ for all $i$,  such that
$
\int_0^\infty \nu_i(x)\ud x < \infty, 
$
as mentioned in the introduction;
\item  
the distribution functions $U_{ij}(\cdot)$  of the jumps when $J_t$
changes from $i$ to $j$ have a continuous density, except possibly at 0,
  so that $U_{ij}(x)=U_{ij} (0) + \int_0^x \mu_{ij}(u) \ud u$, for all
  $i$, $j$, 
in particular, $U_{ii}(0)=1$ and $\mu_{ii}(x) =0$ for all
  $x >0$.
\end{itemize}
The last two assumptions are made to simplify our analysis and are not
much restrictive:  if the
jumps should have discrete components, one would replace the Riemann
integrals by Lebesgue-Stieltjes integrals.

The asymptotic drift $\kappa$ plays an important role, as we see in
Theorem \ref{t:FofG} below.  It is given by
\begin{align}
\kappa & = \pi\tr F'(0) \ones 
   \label{e:stdrift}
 = \pi\tr \left\{ \Delta_a  + \int_0^\infty x \Delta_\nu(x) \ud x 
   + \int_0^\infty x (Q \circ \mu(x))  \ud x \right\} \ones,
\end{align}
where $\pi$ is the stationary probability vector of $Q$: $\pi\tr Q =
0$, $\pi\tr \ones = 1$, and it is such that 
$\lim_{t \rightarrow \infty} X_t = +\infty$ if $\kappa >0$, 
$\lim_{t \rightarrow \infty} X_t = -\infty$ if $\kappa <0$, 
and
$\limsup_{t \rightarrow \infty} X_t = +\infty$, 
$\liminf_{t \rightarrow \infty} X_t = -\infty$ if $\kappa = 0$.

The matrix $G$, such that $(e^Gx)_{i,j}$ is the probability
that the phase is $j$ upon the first visit to level $-x$, given that the process
starts from level 0, in phase $i$ at time 0,
is characterised as follows.

\begin{theorem}
   \label{t:FofG}
Assume that $Q$ is an irreducible generator and that $\sigma_i >0$ for all $i$.  The
matrix $G$  is a
solution of the equation $F(G)=0$, with
\begin{equation}\label{eq:mateq}
\begin{aligned}
F(Y) & = \Delta_a Y + \frac{1}{2} \Delta_\sigma ^2 Y^2 +
\int_0^\infty \Delta_\nu(x) ( e^{Yx} -I) \ud x 
\\   
& \qquad
+ Q\circ U(0)+{\int_0^\infty} (Q \circ \mu(x)) e^{Y x} \ud x.
\end{aligned}
\end{equation}
If $\kappa \leq  0$, $G$ is a generator and is the unique solution in the set  of
real matrices of order $n$ having a simple
eigenvalue equal to 0 and $n-1$ eigenvalues with strictly negative
real parts, if $\kappa >0$, $G$ is a subgenerator and is the unique
solution in the set of
matrices having $n$
eigenvalues with strictly negative real parts.
\end{theorem}
\begin{proof}
  Breuer~\cite{breue08} and D'Auria {\it et al.} \cite{dikm10} analyse
  \levy processes with negative jumps and the adaptation of their
  results to positive jumps is immediate.  As such, the claim that
  $F(G)=0$ is a special case of \cite[Theorem 1]{breue08} and
  \cite[Theorem 2]{dikm10}, uniqueness is proved in \cite{dikm10}.
\end{proof}

\begin{remark} \em
   \label{r:irreducible}
In addition to the stated properties, the matrix $G$ is irreducible.
This is a direct consequence from the assumption that $Q$ is
irreducible, and a formal argument, easy to follow, is given in Proposition~\ref{p:irreducible}.
\end{remark}

We note for later reference that (\ref{eq:mateq}) is equivalent to
\begin{equation}\label{eq:mateq2}
\begin{aligned}
F(Y) = &\frac12 \Delta_{\sigma^2} Y^2 + \Delta_a Y + \int_0^\infty \Delta_\nu(x)\left( 
e^{Yx}-I\right) \ud x \\
 +& Q + \int_0^\infty (Q\circ \mu(x)) (e^{Yx} -I)  \ud x,
\end{aligned}
\end{equation}
because
\begin{equation}
   \label{eq:mateq3}
Q\circ U(0)  +\int_0^\infty (Q\circ \mu(x)) \ud x
= Q \circ (\ones\tr \ones) = Q.
\end{equation}

Relying on this characterization of $G$, in the next sections we will provide the designa and analysis of algorithms for the numerical computation of $G$.

\section{A quadratic matrix equation}\label{sec:ex}
We intend to connect with earlier work on algorithms for quadratic
matrix equations (Bini {\it et al.} \cite{blm:book}) in this section and nonsymmetric
algebraic Riccati equations (Bini {\it et al.} \cite{bim:book}) in
Section \ref{s:riccati}.  
We set $W=I+\tau Y$, where $\tau>0$ is fixed.  
 Replacing $Y$ with $\tau^{-1}(W-I)$ in (\ref{eq:mateq}) and
 multiplying by $2\tau^2$, yield the equivalent matrix equation in the
 unknown $W$ 
\begin{equation}\label{eq:mateqw}
B_1 W^2+B_0(\tau)W+B_{-1}(\tau,W)=0,
\end{equation}
where
\begin{equation}\label{eq:coeff}
\begin{aligned}
 &B_1  =\Delta_{\sigma^2},\quad 
  B_0(\tau) = 2(\tau\Delta_a  - \Delta_{\sigma^2}),\\
& B_{-1}(\tau,W)  =  \Delta_{\sigma^2}-2\tau\Delta_a  + 2\tau^2 \left( H(\tau,W)
+ K(\tau,W)\right),
\end{aligned}
\end{equation}
with 
\begin{align}
   \label{e:HofW}
 H(\tau,W)  & =\int_0^\infty \Delta_\nu(x)\left(
              e^{\tau^{-1}(W-I)x}-I\right) \ud x,\\
   \label{e:KofW}
 K(\tau,W)  & =Q\circ U(0)+ \int_0^\infty (Q\circ
              \mu(x))e^{\tau^{-1}(W-I)x} \ud x.
\end{align}

The sign of the drift $\kappa$, defined in \eqref{e:stdrift}, determines the
existence of a unique stochastic or substochastic solution of \eqref{eq:mateqw}.

\begin{lemma}
   \label{t:uniqueW}
If $\kappa \leq 0$ and (\ref{eq:mateqw}) has a stochastic solution
$W^*$, then $W^*$ is the unique stochastic solution.  If $\kappa >0$ 
and (\ref{eq:mateqw}) has a substochastic solution
$W^*$, then $W^*$ is the unique substochastic solution.  In both
cases, $G= \tau^{-1}(W^* - I)$.
\end{lemma}
\begin{proof}
  Assume that $\kappa \leq 0$ and that (\ref{eq:mateqw}) has
  stochastic solutions, and let $W^*$ be any of them.  The matrix
  $Y^*= \tau^{-1}(W^* - I)$ is a solution of~(\ref{eq:mateq}), it is a
  generator and, by Corollary~\ref{t:FofG}, $Y^* = G$.  In
  consequence, such a stochastic solution would be unique.  A similar
  argument holds in case $\kappa > 0$.
\end{proof}

Thus, our objective is to impose
constraints on $\tau$ which guarantee that~(\ref{eq:mateqw})  has a stochastic or a
substochastic solution.
Firstly, we note that $B_1$ is a diagonal matrix with positive
diagonal entries.  Secondly, $B_0(\tau)$ is a diagonal matrix and its
diagonal entries are negative if
\begin{equation}\label{eq:taus}
0<\tau<\min_{a_i>0}\left\{\frac{\sigma_i^2}{a_i}\right\}.
\end{equation}
Finally, we need $B_{-1}(\tau,W)\ge 0$ for any
(sub)stochastic matrix $W$, and this requires to proceed through a few
steps.  

Step 1: we readily
conclude that the off-diagonal entries are nonnegative.   This is
true of the term  $H(\tau, W)$ since $e^{-M} \geq 0$ for any
Z-matrix $M$.  For $K(\tau, W)$, we recall that  $\mu_{ii}(x) = 0$ for all $i$, 
$Q \circ \mu(x) \geq 0$ for all $x$, and so the integral in
(\ref{e:KofW}) is nonnegative.  Therefore, $K(\tau,W) \geq Q \circ
U(0)$, and 
\begin{equation}\label{e:koftau}
\begin{aligned}
   \nonumber
K_{ij}(\tau, W)  & \geq  Q_{ij} U_{ij}(0) \geq 0  \qquad \mbox{for $i \not= j$},
\\
K_{ii}(\tau, W) & \geq q_{ii} \qquad \mbox{for all $i$.}
\end{aligned}
\end{equation}
This concludes the argument about
the off-diagonal entries, and we are left with
the diagonal entries of  $B_{-1}(\tau,W)$.

Step 2:  if $Y$ is a (sub)generator, then $e^{Yx}$ is (sub)stochastic for
any $x\ge 0$.  
Hence, if $W$ is a (sub)stochastic matrix and $Y = \tau^{-1}(W-I)$, then
\begin{align}
   \nonumber  
K(\tau, W) \ones &  = Q\circ U(0)\ones  +\int_0^\infty (Q\circ
                   \mu(x))e^{Yx}\ones \ud x
\\
\label{eq:qphi1}
  & \leq  Q\circ U(0)\ones  +\int_0^\infty (Q\circ
                   \mu(x))\ones \ud x
\\
  \nonumber 
  & = Q \ones  \qquad \mbox{by (\ref{eq:mateq3})}
   \nonumber 
\\
  \nonumber 
   & = 0,
\end{align}
which implies that  $K(\tau,W)$ is a (sub)generator.

Step 3: concerning the diagonals entries of 
$H(\tau,W)$, since 
$e^{\tau^{-1}x(W-I)}-I\ge 
-I$,
we have
\[
H_{ii}(\tau,W) \ge -\int_0^\infty \nu_i(dx)
\]
and we get from (\ref{e:koftau})
\begin{equation}\label{eq:taus2}
(B_{-1}(\tau,W))_{i,i}\ge \sigma_i^2-2\tau a_i+2\tau^2\left(q_{i,i}
  -\int_0^\infty \nu_i(dx)\right)  \qquad \mbox{for all $i$.}
\end{equation}
Since $q_{ii}<0$, the leading coefficient of the quadratic polynomial
in the right-hand side of \eqref{eq:taus2} is negative, while the
constant coefficient is positive. Therefore the polynomial has two
real roots of opposite signs.

A more accurate bound may
be obtained as follows: as $W \geq 0$, therefore $e^{\tau^{-1}W x} \geq I$ and 
\[
H_{ii}(\tau,W)  \ge 
    \int_0^\infty \nu_i(x)(e^{-\tau^{-1}x}-1)  \ud x
 \ge -\tau^{-1}\int_0^\infty x \nu_i(x) \ud x.
\]
If the latter integral is finite, \eqref{eq:taus2} may be replaced by
\begin{equation}
  \label{eq:taus3}
  (B_{-1}(\tau,W))_{i,i}\ge
\sigma_i^2-\tau \left(2 a_i+ \int_0^\infty x \nu_i(dx)\right)+
2\tau^2 q_{i,i}  \qquad \mbox{for all $i$.}
\end{equation}
Like in \eqref{eq:taus2}, the leading coefficient of the quadratic
polynomial is negative, while the constant coefficient is positive,
and the polynomial has two real roots of opposite signs. Moreover, the
larger the integral in \eqref{eq:taus3}, the larger the positive root
of the polynomial.

We denote as $\tau_i$ the positive root of the polynomial either in
\eqref{eq:taus2} or in~\eqref{eq:taus3}, and we know that 
$(B_{-1}(\tau,W))_{i,i}\ge 0$ if $0 < \tau < \tau_i$.

\begin{theorem}\label{th:1}
Assume that  $W$ is a stochastic or substochastic matrix and that 
\begin{equation}\label{eq:tauc}
0<\tau<\tau^*,~~~\tau^*=\min \left\{ \min_{a_i>0}\left\{\frac{\sigma_i^2}{a_i}\right\},
  \min_i \tau_i \right\}. 
\end{equation}
The matrices $B_{-1}(\tau,W)$, $-B_0(\tau)$ and $B_1$ are nonnegative and such that 
$(B_{-1}(\tau,W)+B_0(\tau)+B_1)\ones\le  0$.
\end{theorem}
\begin{proof}
If $\tau$ satisfies \eqref{eq:tauc}, we have proved that the three matrices  $B_{-1}(\tau,W)$, 
$-B_0(\tau)$ and $B_1$ are nonnegative in the discussion above.
From their definition \eqref{eq:coeff}, we have
\[
\frac12 \tau^{-2}(B_{-1}(\tau,W)+B_0(\tau)+B_1)=
H(\tau,W)+K(\tau,W).
\]
We know from \eqref{eq:qphi1}  that $K(\tau,W)\ones\le 0$.  As
$e^{\tau^{-1}x(W-I)}\ones\le \ones$, it results from~(\ref{e:HofW})
that $H(\tau,W)\ones\le 0$.   This concludes the proof.
\end{proof}

Assume that $\tau$ satisfies (\ref{eq:tauc}).  The matrix $-B_0(\tau)$
has a nonnegative inverse and the matrices 
\begin{equation}\label{eq:wtb}
\wtb_{-1}(\tau,W)=-B_0(\tau)^{-1}B_{-1}(\tau,W), \qquad
\wtb_{1}(\tau)=-B_0(\tau)^{-1}B_{1}
\end{equation}
are nonnegative, with
\begin{equation}\label{eq:sto}
\wtb_{-1}(\tau,W) + \wtb_1(\tau)\ones\le \ones,
\end{equation}
by Theorem~\ref{th:1}.  We now focus our attention on the equation
\begin{equation}\label{eq:w1}
W=\wtb_{-1}(\tau,W)+\wtb_1(\tau) W^2,
\end{equation}
which  is equivalent to \eqref{eq:mateqw}.

\begin{lemma}\label{lem:1}
Assume that $\tau$ satisfies the inequalities \eqref{eq:tauc} and that
$W_1$, $W_2$ are (sub)stochastic matrices.  If $W_1\le W_2$, then $\wtb_{-1}(\tau,W_1)\le \wtb_{-1}(\tau,W_2)$.
\end{lemma}

\begin{proof}
If $W_1$ and $W_2$ are substochastic, then $B_{-1}(\tau,W_1)$ and
$B_{-1}(\tau,W_2)$ are well defined and nonnegative.  
The statement  follows from $-B_0(\tau)^{-1} \geq 0$ and from
$e^{\tau^{-1}(W_1-I)x}\le e^{\tau^{-1}(W_2-I)x}$ for any $x\ge 0$.
\end{proof}

\begin{theorem}\label{th:conv}
Assume that $\tau$ satisfies the inequality \eqref{eq:tauc}. Define the sequence
\begin{equation}\label{eq:fi}
W_{k+1}= \wtb_{-1}(\tau, W_k) + \wtb_1(\tau) W_k^2, \qquad k=0,1,\ldots
\end{equation}
with $W_0=0$.  Then $0\le W_{k}\le W_{k+1}$ and
$W_{k+1}\ones\le \ones$ for all~$k$.  There exists
$\wmin=\lim_{k\to\infty} W_k$.  The matrix $\wmin$ is (sub)stochastic
and is the minimal nonnegative solution to \eqref{eq:mateqw}.

If $\kappa \leq 0$, $\wmin$ is the unique stochastic solution, if
$\kappa > 0$, $\wmin$ is the unique substochastic solution.
In both cases, $G=\tau^{-1}(\wmin - I)$.
\end{theorem}
\begin{proof}
The first statement is proved by induction on $k$.   If $k=0$, then $W_0=0$ and
$W_1=\wtb_{-1}(\tau, 0)$, $W_1$ is nonnegative by Theorem~\ref{th:1};
also,  \eqref{eq:sto} implies that
$W_1\ones\le \ones$. 
Assume now that $0\le W_{k}\le W_{k+1}$ and $W_{k+1}\ones\le \ones$ for a
given $k \geq 0$.  
This implies that 
$\wtb_{1}(\tau) W_k^2\le \wtb_{1}(\tau) W_{k+1}^2$ and, by  Lemma~\ref{lem:1}, that
$0\le \wtb_{-1}(\tau, W_k)\le \wtb_{-1}(\tau, W_{k+1})$.  Therefore, 
$0\le W_{k+1}\le W_{k+2}$.  Furthermore, by the 
inductive assumption and by \eqref{eq:sto}, we find that
$W_{k+2}\ones\le \left( \wtb_{-1}(\tau,
  W_{k+1})+\wtb_1(\tau)\right)\ones\le \ones$.

Since the sequence $\{ W_k\}_k$ is monotone non decreasing and bounded
from above, there exists $\wmin=\lim_{k\to\infty} W_k$, solution of
(\ref{eq:w1}), and therefore of \eqref{eq:mateqw}, by continuity.
We show by contraddiction that $\wmin$ is the minimal nonnegative
solution. If $0\le \widehat{W}\le \wmin$ is another solution, one
proves by induction that ${W}_k\le \widehat{W}$ for any
$k=0,1,2\ldots$, therefore $\wmin=\lim_k W_k\le \widehat{W}$, hence
$\wmin=\widehat{W}$.
The last statement directly follows from Lemma \ref{t:uniqueW}.
\end{proof}


\section{A nonsymmetric algebraic Riccati equation}
   \label{s:riccati}

In Section \ref{sec:ex}, we replace $Y$ by the (sub)stochastic matrix
$W$  through an artificial parameter $\tau$, for which we determine
the constraint (\ref{eq:tauc}).  In this section, we decompose $G$ as
the sum
\begin{equation}\label{e:GpsiB}
G  = - \Delta_b + \Delta_b \Psi 
\end{equation}
where $\Delta_b$ and $\Psi$ have a physical meaning.
This is inspired from a similar decomposition for fluid queues, see
Latouche and Nguyen~\cite[Equations (13) and (15)]{ln17}.

If the process starts from phase $i$ at time 0, it behaves like a
Brownian motion with parameters $a_i$ and $\sigma_i$ until the random
time $T$ where either there is a change of phase (the rate is
$|q_{ii}|$) or there is a jump in the level (the rate is
$\rho_i = \int_0^\infty \nu_i(z) \, \ud z$).  That random variable $T$
is exponentially distributed, with parameter
$\lambda_i = \rho_i + |q_{ii}|$. We condition on the value of the level immediately after time $T$.

It is known (Sato~\cite[Section 45]{sato99}) 
that over an exponential interval, the minimum
$m(T) = \min\{X(s): 0 \leq s \leq T\}$ and the difference $X(T)-m(T)$
are independent and exponential random variables, with respective
parameters
\begin{equation}\label{eq:bi}
b_i = (\sqrt{a_i^2 + 2 \lambda_i \sigma_i^2} +
a_i) / \sigma_i^2
\quad \mbox{and} \quad
c_i = b_i - 2 a_i/ \sigma_i^2,   \qquad i=1,\ldots,n.
\end{equation}

We analyze the trajectory of the process over the interval $[0,T]$,
and find
\begin{equation}\label{e:GpsiA}
\begin{aligned}
e^{Gx}  & =  {e^{-\Delta_b x}} + \int_0^x {\Delta_b e^{-\Delta_b v} \,
          \ud v} \int_0^\infty {\Delta_c e^{-\Delta_c y} \, \ud y} \times
 \\  
  & \quad \big\{\int_0^\infty {\diag(\whP) \Delta_f(z)  e^{G(x+y+z-v)} \, \ud z} \,
 + ((\whP - \diag(\whP)) \circ U(0))  \, e^{G(x+y-v)}
  \\ 
  & \quad +
  \int_0^\infty {((\whP - \diag(\whP)) \circ \mu(z)) e^{G(x+y+z-v)} \,\ud z} \,
\big\}
\end{aligned}
\end{equation}
where
$\whP = \Delta_{\lambda}^{-1} Q + I$ is the phase transition matrix at time $T$,
and 
$f_i(z)= \nu_i(z) /  \rho_i$ is the density of the jumps if there is
no change of phase. 

The justification of (\ref{e:GpsiA}) goes as follows, recall that the exponential on the left is the transition matrix from level 0 to
level $-x$.
\begin{itemize}
\item 
The first term is the probability that $|m(T)| > x$. In that
  case, the level has reached $-x$ before there is any change of phase.
\item
The double integral is the probability that $m(T)= -v$ with $0
< v < x$, and that $X(T)-m(T) = y$ with $y > 0$.
\item 
The expression in brackets is the probability that at time $T$ there is a
jump of size $z \geq 0$, so that the fluid has to go down by a total of
$x-v+y+z$ units.
\end{itemize}
We take the derivative with respect to $x$, evaluated at 0, and obtain
(\ref{e:GpsiB}),  where
\begin{equation}   \label{e:Psi}
\begin{aligned}
\Psi  = & \int_0^\infty {\Delta_c e^{-\Delta_c y} \, \ud y} \, \big\{  
   \int_0^\infty {\diag(\whP) \Delta_f(z) \, \ud z} \,
    e^{G(y+z)}
  \\
& \quad 
 + (\whP - \diag(\whP)) \circ U(0)  \, e^{Gy}
+ \int_0^\infty ((\whP - \diag(\whP)) \circ \mu(z) ) \,\ud z \, e^{G(y+z)}
\big\},
\end{aligned}
\end{equation}
and $b=(b_i)_i$ and $c=(c_i)_i$ are defined in \eqref{eq:bi}.   After
some algebraic manipulation, \eqref{e:Psi} is shown to be equivalent
to \eqref{eq:mateq} with the substitution \eqref{e:GpsiB}.

\begin{proposition}
   \label{p:irreducible}
If the matrix $Q$ is irreducible, then the matrices $\Psi$, $G$ and
$\wmin$ are irreducible.
\end{proposition}

\begin{proof} 
  The proof is adapted from the argument in Van Lierde {\it et
    al.}~\cite[Lemma 3.2]{vldssl06}. 
   By construction, $\Psi$ is a (sub)stochastic matrix: its entry
   $\Psi_{ij}$ is the probability that, starting at time 0 from level
   0 in phase $i$ with a jump to level $y+z$, the process is in phase $j$ upon return
   to level 0.  

As the transition matrix $Q$ is irreducible, there
   exists a path $i=i_1$, $i_2$, \ldots , $i_\ell=j$ with $Q_{i_k,
     i_{k+1}} >0$ for $k=1$, \ldots, $\ell-1$, and there is a strictly positive
   probability that the process moves through the phases $i_1$ to
   $i_\ell$ on its trajectory from level $y+z$ to level 0.  Therefore,
   $\Psi_{ij} > 0$ for all $i$ and $j$, and it follows from Equation
   (\ref{e:GpsiB}) and Theorem \ref{th:conv}, respectively, that $G$
   and $\wmin$  are irreducible as well.
\end{proof}

We replace $G$  with its expression  (\ref{e:GpsiB})  in $F(G)=0$
and, using (\ref{eq:bi}),  obtain
\begin{equation} \label{e:Ge}
\begin{aligned}
0 & = \frac{1}{2} \Delta_\sigma ^2 \Delta_b \Psi \Delta_b \Psi
- \frac{1}{2} \Delta_\sigma ^2 (\Delta_b\Delta_c \Psi  + \Delta_b \Psi
    \Delta_b)  
+\frac{1}{2} \Delta_\sigma ^2 \Delta_b\Delta_c
\\   
& \qquad
+ Q \circ U(0)
+ \int_0^\infty \Delta_\nu (x)  (e^{Gx}-I)  \ud x
+ \int_0^\infty (Q \circ \mu(x)) e^{G x} \, \ud x.
\end{aligned}
\end{equation}
We readily see that
\begin{equation}
   \label{e:Deltabc}
\frac{1}{2} \Delta_\sigma ^2 \Delta_b\Delta_c= \Delta_\lambda 
= { \Delta_\rho - \Delta_q  }
\end{equation}
{where $\Delta_q = \diag(q_{ii})$ (recall that $U_{ii}(0) = 1$
  for all $i$),}
so that  (\ref{e:Ge}) becomes
\begin{align}
  \nonumber  
0 & = \frac{1}{2} \Delta_\sigma ^2 \Delta_b \left\{
\Psi \Delta_b \Psi - \Delta_c \Psi  - \Psi     \Delta_b  \right\}  
+ C(\Psi)
\end{align}
where 
\[
C(\Psi) = (Q - \Delta_q) \circ U(0)
+ \int_0^\infty \Delta_\nu (x) \, e^{(\Delta_b \Psi - \Delta_b) x}
\ud x
+ \int_0^\infty (Q \circ \mu(x)) e^{(\Delta_b \Psi - \Delta_b) x} \, \ud x.
\]
The matrix $C(\Psi)$ is non-negative because both $Q-\Delta_q$ and
$Q\circ \mu(x)$ have zeros on the diagonal, and non-negative
off-diagonal elements.

To prove the next lemma, we follow the same argument as for Lemma
\ref{t:uniqueW}.

\begin{lemma}
   \label{t:uniquePsi}
If $\kappa \leq 0$, there is a unique stochastic matrix $\Psi$
solution of
\begin{align}
  \nonumber  
0 & = \frac{1}{2} \Delta_\sigma ^2 \Delta_b \left\{
X \Delta_b X - \Delta_c X  - X     \Delta_b  \right\}  
+ C(X).
\end{align}
If $\kappa >0$, there is a unique substochastic solution $\Psi$.  In
both cases, $G= -\Delta_b + \Delta_b \Psi$. 
\qed
\end{lemma}

The matrix $\Delta_b\Psi$ solves the equation
\begin{equation}\label{eq:narex}
X^2-\Delta_cX-X\Delta_b+2\Delta_\sigma^{-2}\widehat C(X)=0
\end{equation}
in the unknown $X$, where
\begin{equation}\label{eq:chat}
\widehat C(X)= (Q - \Delta_q) \circ U(0)
+ \int_0^\infty \Delta_\nu (x) \, e^{(X - \Delta_b) x}  \ud x
+ \int_0^\infty (Q \circ \mu(x)) e^{(X - \Delta_b) x} \, \ud x.
\end{equation}

If $W$ is a given matrix and if we interpret $\widehat C(W)$ as a
constant coefficient, equation \eqref{eq:narex} can be transformed
into a
nonsymmetric algebraic Riccati equation (NARE)
\begin{equation}\label{eq:narew}
 S^2-\Delta_c S - S \Delta_b
+2 \Delta_\sigma^{-2} \widehat C(W)=0.
\end{equation}

\begin{proposition}\label{prop:mnare}
If $W\ge 0$ and $W\ones \le b$, then the matrix
\[
M=\begin{bmatrix}
\Delta_b & -I\\
-2\Delta_\sigma^{-2} \widehat{C}(W) & \Delta_c
\end{bmatrix}
\]
is an irreducible M-matrix and it is singular if $W\ones=b$. 
Equation \eqref{eq:narew} has a minimal nonnegative solution $\smin(W)
> 0$ and $\smin(W)\ones\le b$.
Finally,  $\smin(W)\ones= b$ if
$W\ones=b$.
 \end{proposition}

\begin{proof}
As $W\ge 0$ and $W\ones\le b$, $W - \Delta_b$ is a generator and $e^{(W-
  \Delta_b)x}$ is a nonnegative (sub)stochastic matrix with strictly positive
diagonal.  This implies that $\widehat{C}(W)$ is irreducible
nonnegative, and $M$ is, therefore, an irreducible Z-matrix.
We show next that 
\[
\widetilde M=M \begin{bmatrix}
I & 0 \\ 0 & \Delta_b
\end{bmatrix}=
\begin{bmatrix}
\Delta_b & -\Delta_b \\  
-2\Delta_\sigma^{-2} \widehat{C}(W) & \Delta_c\Delta_b
\end{bmatrix}
\]
 satisfies the property $\widetilde M \ones\ge 0$, and $\widetilde
 M\ones =0$ if $W\ones=b$. 
Set 
 $
 \begin{bmatrix}
 y_1 \\ y_2
 \end{bmatrix}
 = \widetilde M \begin{bmatrix}
 \ones \\ \ones
 \end{bmatrix}
 $. Clearly $y_1=0$, and we only need to deal with $y_2$.
As $e^{(W - \Delta_b)t}$ is substochastic, 
\begin{equation}\label{e:C1}
\begin{aligned}
\widehat C(W)\ones  & \le 
 (Q - \Delta_q) \circ U(0)\ones
+ \int_0^\infty \Delta_\nu (\ud x) \,\ones  
+ \int_0^\infty (Q \circ \mu(x))  \, \ud x \ones
\\ 
 & 
 = Q \ones + \Delta_\lambda \ones=\Delta_\lambda \ones,
\end{aligned}
\end{equation}
and $y_2\ge -2\Delta_\sigma^{-2}\Delta_\lambda \ones + \Delta_c
\Delta_b \ones=0$.  
 If $W\ones=b$, then $e^{(W-\Delta_b)t}$ is stochastic, so that
 $\widehat C(W)\ones = \Delta_\lambda \ones$ and 
$y_2=0$. 
This  implies by Theorem \ref{thm:mm}
 that $M$ is an (irreducible) M-matrix, singular if $W \ones =b$  
 and,  by Bini {\it et
   al.}~\cite[Theorem 2.9]{bim:book}, that there exists a 
minimal nonnegative solution which is strictly positive.

Now we show that $S(W)\ones\le b$.   By the argument in the proof of
\cite[Theorem 2.9]{bim:book},  $\smin(W)=\lim_{k\to\infty}S_k(W)$,
where $S_0(W)=0$ and 
\begin{equation}\label{eq:recric}
\Delta_c S_{k+1}(W)+S_{k+1}(W)\Delta_b=S_k(W)^2+2\Delta_\sigma^{-2}\widehat C(W), ~~k=0,1,\ldots
\end{equation}
Moreover, $0\le S_k(W)\le S_{k+1}(W)$.  We prove by induction on $k$,
that $S_k(W)\ones\le b$ for any $k\ge 0$, so that in the limit as
$k\to\infty$, $\smin(W)\ones\le b$.  If $k=0$, $S_0(W)=0$ and the
inequality holds. Assume that $S_k(W)\ones\le b$.   By monotonicity of
the sequence $\{S_k(W)\}$ we have 
\begin{align*}
\Delta_c S_{k+1}(W)\ones & + S_{k+1}(W)\Delta_b\ones  \leq S_k(W)^2\ones
+\Delta_c \Delta_b\ones \qquad \mbox{by (\ref{e:Deltabc}, \ref{e:C1})} 
\\
\hbox{\rm or \ } \Delta_c S_{k+1}(W)\ones & \leq  - S_{k+1}(W)b  + S_k(W)b
+\Delta_c b,
\end{align*}
that is, $S_{k+1}(W)\ones\le b$.
Finally, if $W \ones = b$, we have seen that $\widehat C(W)\ones =
 \Delta_\lambda \ones$ and we obtain from (\ref{eq:narew}) that
\begin{align*}
0 & = \smin(W)^2 \ones - \Delta_c \smin(W) \ones - \smin(W) b + 2
     \Delta_\sigma^2 \Delta_\lambda \ones
 \\ 
   & = (\smin(W) - \Delta_c)(\smin(W) \ones -b)
    = \smin(W) \ones -b
\end{align*}
since $\smin(W) - \Delta_c$ is nonsingular by \cite[Theorem
2.9]{blm:book}.  This completes the proof.
\end{proof}

\begin{lemma}\label{lem:cx}
If $0\le X\le Y$, $X\ones \le b$ and $Y\ones\le b$, then $\widehat
C(X)\le \widehat C(Y)$ and $\smin(X) \leq \smin(Y)$.
\end{lemma}

\begin{proof}
From \eqref{eq:expdiff}, 
\[
e^{(Y-\Delta_b)x}-e^{(X-\Delta_b)x}=
\int_0^x e^{(X-\Delta_b)(x-s)}(Y-X)e^{(Y-\Delta_b)s} \geq 0
\]
since $Y-X\ge 0$ and $\Delta_b-X$, $\Delta_b-Y$ are
M-matrices. Therefore, from \eqref{eq:chat}, we have $\widehat C(X)\le
\widehat C(Y)$.

To prove the second statement, we recall that
$\smin(W)=\lim_{k\to\infty}S_k(W)$, where $S_k(W)$ is defined in
\eqref{eq:recric}, with $S_0(W)=0$. We may easily show by induction on~$k$
that $S_k(X)\le S_k(Y)$, which implies that $\smin(X)\le \smin(Y)$.
\end{proof}

The next theorem is a direct consequence of Lemma \ref{t:uniquePsi},
Proposition~\ref{prop:mnare} and \eqref{e:GpsiB}.
\begin{theorem}
   \label{t:smin}
  If $W=\Delta_b \Psi$, then the minimal nonnegative solution of
  \eqref{eq:narew} is $\smin(W)=\Delta_b \Psi$. Moreover, $\Psi>0$ and the off-diagonal entries of $G$ are strictly positive.
\qed
\end{theorem}

\section{Computational methods}\label{sec:com}

The recurrence \eqref{eq:fi}, with $W_0=0$, provides a fixed point iteration for computing the  
 minimal nonnegative solution to \eqref{eq:mateqw}.  Its convergence,
 however, may be slow and we develop faster algorithms in this section.

\subsection{$U$-based iteration}
   \label{s:ubased}
We introduce a fixed point iteration which resembles the $U$-based
iteration for  solving matrix equations arising in M/G/1-type Markov
chains (Bini {\it et al.}~\cite[Section 6.2]{blm:book}).  Take $W_0=0$
and define
\begin{equation}\label{eq:fixu}
W_{k+1}= (I- \wtb_1(\tau) W_k )^{-1}\wtb_{-1}(\tau, W_k),~k=0,1,\ldots,
\end{equation}
where the matrix $I- \wtb_1(\tau) W_k$ is invertible, as well as $I- \wtb_1(\tau) W_{\min}$,
in view of the arguments in \cite[Section 6.2]{blm:book}.
By the same arguments as in the proof of Theorem~\ref{th:conv}, one may
show that the sequence $W_k$ converges monotonically to
the minimal nonnegative solution $\wmin$ of \eqref{eq:mateqw}.

 A pseudocode of this iteration  is reported in Algorithm \ref{alg:ubased}.
 
\begin{algorithm} 
\caption{$U$-based algorithm} 
\label{alg:ubased} 
\begin{algorithmic}[1] 
    \REQUIRE The parameters $a,\sigma,\nu,\mu,Q,U$ defining \eqref{eq:mateq},  a value of $\tau$ satisfying \eqref{eq:tauc}, a starting approximation $W_0$, and a tolerance $\epsilon>0$. 
    \ENSURE An approximation $G$ to the solution of \eqref{eq:mateq}.
    \STATE Compute $W_{1}= (I- \wtb_1(\tau) W_0 )^{-1}\wtb_{-1}(\tau, W_0)$,
    where $\widetilde B_1(\tau)$ and $\widetilde B_{-1}(\tau, W_0)$ are defined in \eqref{eq:wtb}.
    \STATE Compute $err=\| W_{1}-W_0\|_\infty$   and set $k=1$. 
    \WHILE{$err>\epsilon$}
    \STATE Compute $W_{k+1}= (I- \wtb_1(\tau) W_k )^{-1}\wtb_{-1}(\tau, W_k)$.
    \STATE Compute $err=\| W_{k+1}-W_k\|_\infty$ and set $k=k+1$.
     \ENDWHILE
     \STATE Set $G=\tau^{-1}(W_{k}-I)$.
\end{algorithmic}
\end{algorithm}

We provide a convergence analysis for the sequence \eqref{eq:fixu}, a
similar analysis may be repeated for the sequence \eqref{eq:fi}.
Define the error at step $k$ as
\[
E_k=\wmin-W_k,~~k=0,1,\ldots
\]
As  $W_k\le \wmin$, we have  $E_k\ge 0$ for any $k\ge 0$ and, in
particular,  $\| E_k\|_\infty=\| E_k\ones\|_\infty$.
By subtracting \eqref{eq:fixu} from
\[
\wmin= (I- \wtb_1(\tau) \wmin )^{-1}\wtb_{-1}(\tau, \wmin)
\]
we find that
\begin{equation}\label{eq:ek}
E_{k+1}=(I-\wtb_1(\tau)\wmin)^{-1}\left(-2\tau^2 B_0(\tau)^{-1}  A_k+\wtb_1(\tau) E_k W_{k+1}\right),
\end{equation}
where
\begin{align}
   \nonumber
A_k  & =H(\wmin,\tau)-H(W_k,\tau)+K(\wmin,\tau)-K(W_k,\tau)
\\
   \label{eq:sk}
  & = \int_0^\infty \Delta_\nu(x)\Gamma_k(x)  \, \ud x
+
\int_0^\infty (Q\circ \mu(x)) \Gamma_k(x)  \, \ud x
\end{align}
by \eqref{eq:coeff},  with
\[
\Gamma_k(x)
=e^{\tau^{-1}x(\wmin-I)}-e^{\tau^{-1}x(W_k-I)}
   = \int_0^{\tau^{-1}x} e^{(W_k-I)(\tau^{-1}x-s)}E_k e^{(\wmin-I)s}
    \ud s
\]
by \eqref{eq:expdiff}.   The integrand is nonnegative since $E_k\ge 0$,
and $e^{(\wmin-I)s}\ones\le \ones$ since $\wmin\ones\le\ones$.  Therefore,
\[
\Gamma_k(x)\ones
\le\left(\int_0^{\tau^{-1}x}
                   e^{(W_k-I)(\tau^{-1}x-s)} ds\right) E_k\ones
\le \left(\lim_{k \to \infty} \int_0^{\tau^{-1}x}
                   e^{(W_k-I)(\tau^{-1}x-s)} ds\right) E_k\ones.
\]
Since $\{ W_k\}_k$ converges monotonically to $\wmin$, then
we get
 \begin{equation} \label{eq:v}
 \Gamma_k(x)\ones\le \int_0^{\tau^{-1}x}   e^{(\wmin-I)(\tau^{-1}x-s)} \ud s \,  E_k\ones
= \tau^{-1} \int_0^x e^{Gs} \ud s  \, E_k\ones
\end{equation}
by Theorem \ref{th:conv}.  We combine (\ref{eq:sk}, \ref{eq:v}) and write
\begin{equation}
   \label{e:skone}
A_k \ones \leq \tau^{-1} \Lambda E_k \ones,
\end{equation}
where
\begin{equation}\label{eq:lambda}
\Lambda =
\int_0^\infty \Delta_\nu(x) \int_0^x e^{Gs} \,\ud s \,\ud x
+
\int_0^\infty (Q\circ \mu(x))  \int_0^x e^{Gs} \,\ud s \,\ud x.
\end{equation}

Hence, from \eqref{eq:ek}, 
\begin{equation}  \nonumber   
E_{k+1}\ones \le R E_k\ones \leq R^{k+1} E_0 \ones
\end{equation}
where
\begin{equation}   \nonumber   
R=
(I-\wtb_1(\tau)\wmin)^{-1}\left(-2\tau B_0(\tau)^{-1}  
\Lambda
+\wtb_1(\tau) \right),
\end{equation}
and so  $\| E_{k}\ones\| \le \| R^k\| \| E_0\ones\|$ for any vector norm $\|\cdot\|$ where $\| A \|$ is the corresponding induced matrix norm.
Therefore the asymptotic rate of convergence, given by $\lim_k \|
R^k\|^{1/k}$ coincides with $\rho(R)$, and our objective in the remainder
of this subsection is to further analyse $R$.

We replace $\wtb_1(\tau)$ with its expression  \eqref{eq:wtb}, and
obtain
\begin{equation}\label{eq:R0}
R=-(B_0(\tau)+B_1(\tau G+I))^{-1}\left(2\tau \Lambda
+B_1 \right)
\end{equation}
next, we replace $B_0(\tau)$ and $B_1$ with their expressions in
\eqref{eq:coeff}, and obtain
\begin{equation}  \nonumber  
R=(\tau^{-1}\Delta_\sigma^2-2\Delta_a-\Delta_\sigma^{2}G)^{-1}
\left(2\Lambda+\tau^{-1} \Delta_\sigma^{2}\right).
\end{equation}
We see that $R=M_1^{-1}N_1$,with
\begin{equation}
   \label{e:mone}
M_1=\tau^{-1}\Delta_\sigma^2-2\Delta_a-\Delta_\sigma^2G
\quad \mbox{and} \quad
N_1=2\Lambda+\tau^{-1}\Delta_\sigma^{2}.
\end{equation}
Moreover, $M_1-N_1$ is equal to the matrix 
\begin{equation}\label{eq:xi}
\Theta=-2\Delta_a-\Delta_\sigma^{2}G-2\Lambda.
\end{equation}
We show in Theorem \ref{t:theta} below that $\Theta$ is an irreducible
M-matrix and that $\Theta=M_1-N_1$ is a regular splitting. Moreover if $\kappa<0$ then 
$\Theta$ is nonsingular, so that
$\rho(R)<1$ by Theorem \ref{thm:rs}.  Furthermore, since $N_1$ is a
nonincreasing function of $\tau$, the spectral radius of $R$ is a
nonincreasing function of $\tau$, and the larger  $\tau$, the
faster the convergence, provided that $\tau$ satisfies
condition~\eqref{eq:taus}.

\begin{theorem}
   \label{t:theta}
The matrix $\Theta$ is irreducible.
If  $\kappa< 0$, then it is a nonsingular M-matrix. Otherwise, if
$\kappa\ge 0$, then it  is a singular M-matrix. Moreover, $\Theta=M_1-N_1$ is a regular splitting, where $M_1,N_1$ are defined in \eqref{e:mone}.
\end{theorem}

\begin{proof}
  The matrix $\Lambda$ is nonnegative and $G$ is an irreducible
  (sub)generator with strictly positive off-diagonal entries in view of Theorem \ref{t:smin}.  Therefore,
  the off-diagonal entries of $\Theta$ are strictly negative, that is,
  $\Theta$ is an irreducible Z-matrix.  With $\pi>0$ the stationary
  probability vector of $Q$, we only need to prove that
  $\pi^T \Theta>0$ if $\kappa< 0$ and that $\pi^T \Theta=0$ if
  $\kappa\ge 0$, so that, by Theorem \ref{thm:mm}, $\Theta$ is a
  nonsingular M-matrix in the first case, a singular M-matrix in the
  second.

We observe that $\int_0^x e^{Gs} G \ud s = e^{Gx} -I$, and so
\begin{align*}
\Theta G  =& -2 \big\{ \Delta_a G + \frac{1}{2} \Delta_\sigma^2 G^2 
  +  \int_0^\infty \Delta_\nu(x) (e^{Gx} -I) \ud x
  \\
  &  + \int_0^\infty  (Q \circ \mu(x)) (e^{Gx} -I) \ud x  \big\}
 = -2 Q,
\end{align*}
using $F(G)=0$ and (\ref{eq:mateq2}).  This implies that
\begin{equation}
   \label{e:pithetag}
\pi\tr \Theta G = 0.
\end{equation}
If $\kappa >0$, then $G$ is a subgenerator, it is nonsingular and we find
from (\ref{e:pithetag}) that $\pi\tr \Theta = 0$.
If $\kappa \leq 0$, then $G$ is a generator and 
\[
\Theta \ones 
= -2 (\Delta_a \ones + \Lambda \ones)
    = -2 \big\{\Delta_a \ones  + \int_0^\infty x \Delta_\nu(x) \ud x
  + \int_0^\infty  (Q \circ \mu(x)) \ud x  \big\}
\]
which implies that
$
\pi\tr \Theta \ones = - 2 \kappa
$
by (\ref{e:stdrift}) and Corollary \ref{t:FofG}.
Also, by (\ref{e:pithetag}), we have
\[
\mbox{either} \qquad   \pi\tr \Theta = 0 \qquad \mbox{or} \qquad 
   \pi\tr \Theta = \gamma s\tr, \quad \gamma \not= 0, 
\]
 where $s$ is the
steady state vector of $G$, that is, $s\tr G=0$, $s\tr \ones =1$
(recall from Proposition \ref{p:irreducible} that $G$ is irreducible).  
Now, if $\kappa < 0$, then $\pi\tr \ones$ cannot be 0, and we are left
with the option $\pi\tr = \gamma s\tr$ with $\gamma >0$, so that $\pi\tr
\Theta >0$.
Finally, if $\kappa = 0$, then $\pi\tr \Theta$ cannot be equal to
$\gamma s\tr$ with $\gamma \not= 0$, the only possibility is $\pi\tr
\Theta = 0$.

Concerning the last claim, the matrix $N_1$ is clearly nonnegative, the matrix $M_1$ is a Z-matrix and is such that $M_1\ge \Theta$. Therefore, in view of Theorem \ref{thm:mm}, $M_1$ is an M-matrix. $M_1$ is nonsingular since its diagonal entries are strictly larger than the corresponding diagonal entries of $\Theta$. Hence $M_1^{-1}\ge 0$.
\end{proof}

 \subsection{QME-based iteration}
Here we propose a more effective algorithm, that consists in generating a sequence of matrices $\{W_k\}_k$, such that $W_{k+1}$ is the minimal nonnegative solution of the quadratic matrix equation
\begin{equation}\label{eq:qmeit}
W_{k+1}=\wtb_{-1}(\tau, W_k)+  \wtb_1(\tau) W_{k+1}^2,~~k=0,1,\ldots
\end{equation}
 with $W_0=0$.
By following the arguments of Theorem~\ref{th:conv}, one shows that
the sequence $\{W_k\}_k$ monotonically converges to $\wmin$; we do not
repeat the details.

 A pseudocode of this iteration  is reported in Algorithm \ref{alg:qmebased}. The computation of $W_{k+1}$, given $W_k$, at steps \ref{s:w1} and \ref{s:wk} can be performed by using fixed point iterations or cyclic reduction \cite{blm:book}.
 
\begin{algorithm} 
\caption{QME-based algorithm} 
\label{alg:qmebased} 
\begin{algorithmic}[1] 
    \REQUIRE The parameters $a,\sigma,\nu,\mu,Q,U$ defining \eqref{eq:mateq},  a value of $\tau$ satisfying \eqref{eq:tauc}, a starting approximation $W_0$, and a tolerance $\epsilon>0$. 
    \ENSURE An approximation $G$ to the solution of \eqref{eq:mateq}
    \STATE\label{s:w1} Compute $W_{1}$ such that
    $W_{1}=\wtb_{-1}(\tau, W_0)+  \wtb_1(\tau) W_{1}^2$,
    where $\widetilde B_1(\tau)$ and $\widetilde B_{-1}(\tau, W_0)$ are defined in \eqref{eq:wtb}.
 \STATE Compute $err=\| W_{1}-W_0\|_\infty$ and set $k=1$. 
    \WHILE{$err>\epsilon$}
    \STATE\label{s:wk} Compute $W_{k+1}$ such that
$ W_{k+1}=\wtb_{-1}(\tau, W_k)+  \wtb_1(\tau) W_{k+1}^2$.
    \STATE Compute $err=\| W_{k+1}-W_k\|_\infty$ and set $k=k+1$.
     \ENDWHILE
     \STATE Set $G=\tau^{-1}(W_{k}-I)$.
\end{algorithmic}
\end{algorithm}

Now we analyse the speed of convergence.
Define the error $E_k=\wmin-W_k$. By subtracting \eqref{eq:qmeit} from
\eqref{eq:w1}, we find that
\[
(I-\wtb_1(\tau)\wmin)E_{k+1}-\wtb_1(\tau) E_{k+1}W_{k+1}=
-2\tau^2 B_0(\tau)^{-1}A_k,
\]
where $A_k$ is defined in \eqref{eq:sk}. By setting
$P=(I-\wtb_1(\tau)\wmin)^{-1}\wtb_1(\tau)$ and $T_k=-2\tau^2
(I-\wtb_1(\tau)\wmin)^{-1} B_0(\tau)^{-1}A_k$, we successively get
\begin{align*}
E_{k+1}  & =\sum_{j=0}^\infty P^j T_k W_{k+1}^j,
\\
E_{k+1}\ones  &  \le \sum_{j=0}^\infty P^j T_k\ones \, = \, (I-P)^{-1} T_k\ones
 \\
 & =
-2\tau^2 (I-P)^{-1}(I-\wtb_1(\tau)\wmin)^{-1} B_0(\tau)^{-1}A_k\ones
\\
  & = -2\tau^2  (B_0(\tau)+B_1\wmin+B_1)^{-1} A_k\ones
\\
 & = -2 \tau (2\Delta_a + \Delta_\sigma^2 G)^{-1}  A_k \ones
  \qquad \mbox{by \eqref{eq:coeff} and Theorem \ref{th:conv},}
\\
  & \leq -2  (2\Delta_a + \Delta_\sigma^2 G)^{-1}\Lambda E_k\ones,
\end{align*}
where $\Lambda$ is defined in \eqref{eq:lambda}, the last inequality
being shown by the same argument as led us to (\ref{e:skone}).
Therefore, 
\begin{equation}\label{eq:Rhat}
E_{k+1}\ones\le \widehat R E_k\ones,
\quad \mbox{with} \quad
\widehat R=-2(2\Delta_a+\Delta_\sigma^2 G)^{-1}\Lambda.
\end{equation}
Observe that the matrix $\widehat R$ is
independent of $\tau$ and that $\widehat R=M_2^{-1}N_2$, with
\begin{equation}\label{eq:M2}
M_2=-2\Delta_a-\Delta_\sigma^2G
\quad \mbox{and} \quad
  N_2=2\Lambda,
\end{equation}
two matrices such that $M_2^{-1}\ge 0$, $N_2\ge 0$ and $M_2-N_2=\Theta$ defined in~\eqref{eq:xi}. That is, $M_2-N_2$ is a regular splitting of $\Theta$,
different from the splitting $M_1-N_1$ of (\ref{e:mone}).
Since $N_2\le N_1$ for any $\tau>0$, $\rho(\widehat{R})\le \rho (R)$
by Theorem~\ref{thm:rs}. The
inequality is strict if $\widehat R$ is irreducible for the
Perron-Frobenius theorem~\cite{bp:book} or if $\Theta^{-1}>0$ (see
Theorem \ref{thm:rs}).
As a matter of fact, $\lim_{\tau\to\infty} R=\widehat R$ but $\tau$
must satisfy the condition \eqref{eq:taus}, and so the speed of
convergence of the iteration \eqref{eq:fixu} cannot reach the speed of
convergence of \eqref{eq:qmeit}.

\subsection{NARE-based iteration}
Set $S_0=0$ and define the sequence $\{S_k\}$, where $S_{k+1}$ is the
minimal nonnegative solution of the NARE
\begin{equation}\label{eq:narexk}
S_{k+1}^2-\Delta_cS_{k+1}-S_{k+1}\Delta_b+2\Delta_\sigma^{-2}\widehat
C(S_k)=0, \qquad k=0,1,\ldots
\end{equation}
This sequence is well defined and converges to the minimal nonnegative
solution  $\Delta_b \Psi$  of \eqref{eq:narex}, as stated in the next theorem.

\begin{theorem}\label{thm:convnare}
  If $S_0=0$, then \eqref{eq:narexk} has a minimal nonnegative
  solution $S_{k+1}$ for any $k\ge 0$. Moreover $S_k\ones \le b$,
  $0\le S_k\le S_{k+1}$ and $\lim_k S_k=\smin$, where
  $\smin = \Delta_b \Psi$ is the minimal nonnegative solution of
  \eqref{eq:narex} among the nonnegative solutions such that
  $S\ones\le b$.
\end{theorem}

\begin{proof}
The proof is by induction on $k$.  For  $k=0$, as $S_0\ge0$ and
$S_0\ones\le b$,
\eqref{eq:narexk} has by Proposition \ref{prop:mnare} a
minimal nonnegative solution $S_{1}\ge S_0$ such that $S_1\ones\le b$.
Assume that 
$0\le S_{k-1}\le S_{k}$ and $S_k\ones\le b$ for a given $k$. 
By Proposition \ref{prop:mnare} again,
\eqref{eq:narexk} has a minimal nonnegative solution $S_{k+1}=S(S_k)$
such that $S_{k+1}\ones\le b$ and by Lemma \ref{lem:cx}, 
$0\le S_{k-1}\le S_{k}$ implies that $\smin(S_{k-1})\le \smin(S_k)$,
that is, $S_k\le S_{k+1}$. 

The sequence $\{S_k\}_k$ is monotone and bounded, there exists,
therefore, a limit $S^*=\lim_{k\to\infty}S_k$ and $S^*$ solves
\eqref{eq:narex} by continuity of $C(X)$. If there exists another
solution $Y\ge0$ such that $Y\ones\le b$, then one proves by induction
that $S_k\le Y$ for all $k$.  Therefore $S^*\le Y$, $S^* = \smin$ and
$\smin = \Delta_b \Psi$ by Theorem~\ref{t:smin}.  This concludes the
proof.
\end{proof}

 A pseudocode of this iteration  is reported in Algorithm \ref{alg:narebased}. The computation of $S_{k+1}$, given $S_k$, at steps \ref{s:x1} and \ref{s:xk} can be performed by using fixed point iterations or the Structured Doubling Algorithm \cite{bim:book}.
 
\begin{algorithm} 
\caption{NARE-based algorithm} 
\label{alg:narebased} 
\begin{algorithmic}[1] 
    \REQUIRE The parameters $a, \sigma,\nu,\mu,Q,U$ defining \eqref{eq:mateq}, a starting approximation $S_0$, a tolerance $\epsilon>0$.
    \ENSURE An approximation $G$ to the solution of \eqref{eq:mateq}
    \STATE\label{s:x1} Compute $S_{1}$ such that
$S_{1}^2-\Delta_cS_{1}-S_{1}\Delta_b+2\Delta_\sigma^{-2}\widehat C(S_0)=0$, where $b$ and $c$ are defined in \eqref{eq:bi} and $\widehat C(\cdot)$ is defined in \eqref{eq:chat}.   
 \STATE Compute $err=\| S_{1}-S_0\|_\infty$ and set $k=1$. 
    \WHILE{$err>\epsilon$}
    \STATE\label{s:xk} Compute $S_{k+1}$ such that
    $
    S_{k+1}^2-\Delta_cS_{k+1}-S_{k+1}\Delta_b+2\Delta_\sigma^{-2}\widehat C(S_k)=0$. 
    \STATE Compute $err=\| S_{k+1}-S_k\|_\infty$ and set $k=k+1$.
     \ENDWHILE
     \STATE Set $G=S_k-\Delta_b$.
\end{algorithmic}
\end{algorithm}

We perform a convergence analysis of the sequence $\{S_k\}_k$. 
Define the error $\whe_k= \Delta_b \Psi-S_k$. Since $\whe_k\ge 0$ by
the monotonic convergence, we use $\whe_k\ones$ as a measure of the
error. 
By arguments similar to
the ones in the previous sections, we find that 
\[
(\Delta_c - S_{k+1})\whe_{k+1} - \whe_{k+1}
G=2\Delta_\sigma^{-2}(\widehat C(\Delta_b \Psi)-\widehat C(S_k)).
\]
Next, we repeat the argument between (\ref{eq:sk}) and (\ref{e:skone})
and show that 
\[
(\widehat C(\Delta_b \Psi)-\widehat C(S_k))\ones \leq \Lambda \whe_k \ones
\]
where $\Lambda$ is defined in \eqref{eq:lambda}.
As $G$ is  a (sub)generator, $G \ones \leq 0$ and
\begin{equation}\label{eq:eknare}
(\Delta_c - S_{k+1})\whe_{k+1}\ones \leq 2\Delta_\sigma^{-2} \Lambda
\whe_k \ones.
\end{equation}

Observe that 
\[
\Delta_c-S_{k+1} \ge \Delta_c- \Delta_b \Psi =\Delta_c
-\Delta_b-G =-(G + 2\Delta_\sigma^{-2}\Delta_a) = \Delta_\sigma^{-2}
M_2
\]
with $M_2$ defined in \eqref{eq:M2}.
Therefore, we obtain from \eqref{eq:eknare} that
\[ 
M_2 \whe_{k+1}\ones \leq 2\Delta_\sigma^{-2} \Lambda
\whe_k \ones
\]
and, since $M_2$ has a nonnegative inverse,
\[
\whe_{k+1}\ones \le 
-2( \Delta_\sigma^2 Y+2\Delta_a)^{-1}\Lambda \whe_k\ones,
\]
that is, the same inequality  as \eqref{eq:Rhat}. Hence, the
NARE-based iteration has the same asymptotic speed of convergence as
the QME-based iteration.

\begin{remark} \em
Define
\[
\calg(M) = \int_0^\infty \Delta_\nu(x)\left( e^{Mx}-I\right) \ud x 
+
Q\circ U(0)+\int_0^\infty (Q\circ \mu(x))e^{Mx}  \ud x.
\]
The matrix $G$ is the solution of
\[
\Delta_{\sigma^2} Y^2 + 2 \Delta_a Y + 2 \calg(Y) =0
\]
with characteristics given in Corollary~\ref{t:FofG}.  By simple
but tedious manipulations, we may show that $W_{k+1} = I + \tau X$ with
\[
\Delta_{\sigma^2} X^2 + 2 \Delta_a X + 2 \calg(\tau^{-1} (W_k - I)) =0,
\]
and also that $S_{k+1} = \Delta_b + Z$, with 
\[
\Delta_{\sigma^2} Z^2 + 2 \Delta_a Z + 2 \calg(S_k - \Delta_b) =0.
\]
In view of these similar equations, it is not surprising that the
QME- and the NARE-based iterations should have the same asymptotic
rate of convergence.
Moreover, if $S_0$ and $W_0$ are such that $S_0-\Delta_b=\tau^{-1}(W_0-I)$ then for any $k\ge 0$ we have $S_k-\Delta_b=\tau^{-1}(W_k-I)$. That is, the approximation to $G$ obtained by the NARE-based iteration coincides with the approximation obtained by the QME-based iteration.
\end{remark}

\section{Numerical experiments}\label{sec:num}

We have implemented the $U$-, QME- and NARE-based algorithms \ref{alg:ubased}, \ref{alg:qmebased} and \ref{alg:narebased} in Matlab.
 As threshold value  for the stopping condition we have chosen $\epsilon=10^{-14}$.
 The QME and the NARE at steps \ref{s:w1}, \ref{s:wk} of Algorithms
 \ref{alg:ubased} and \ref{alg:narebased}, respectively, have been
 solved by means of Cyclic Reduction (CR) and Structured
 Doubling Algorithm (SDA);  we have used the Matlab functions \texttt{cr} and \texttt{sda} from the book \cite{bim:book}, downloadable from \url{http://riccati.dm.unipi.it/nsare/mfiles.html}.
The integrals defining the matrices $\widetilde B_{-1}(\tau,W)$ in Algorithms \ref{alg:ubased} and \ref{alg:qmebased}, and
 $\widehat C(X)$ in Algorithm \ref{alg:narebased},
have been computed by means of the Matlab function \texttt{integral}, with precision
\texttt{'RelTol',1.e-10}, \texttt{'AbsTol',1.e-12}.

 The experiments have been run on a PC with CPU Intel Core i7-7700, under Ubuntu system, with Matlab version R2018a.
 
 We have compared our algorithms with the algorithms proposed by Simon in \cite{matth17} and by Breuer in \cite{breue08}. 
   The iteration in Simon~\cite[Proposition 4.8]{matth17} is
\begin{equation}   \nonumber 
S_{k+1} (S_k - \Delta_b)
-\Delta_c S_{k+1}+2\Delta_\sigma^{-2}\widehat
C(S_k)=0, \qquad k=0,1,\ldots
\end{equation}
with $S_0=0$.  

The Breuer algorithm consists in a preprocessing stage, where the 
minimal solution $s_i$ of the equation
\[
a_i s + \frac12 \sigma_i^2 s^2 +\int_0^\infty (e^{sx}-1)\nu_i(x)dx=|q_{ii}|,~~i=1,\ldots,n
\]
is computed. The iteration, with $G_0=0$, is defined by
\[
e_i\tr G_{k+1}=-\xi_i e_i\tr+e_i\tr\left( Q\circ(U(0)-I) +
\int_0^\infty (Q\circ \mu(x)) e^{G_k x} dx \right) L_i^*(G_k),
\]
for $i=1,\ldots,n$ and $k=0,1,\ldots$,
where $e_i$ is the $i$-th column of the identity matrix, $\xi_i=|s_i|$, $L_i^*(V)=(\xi_i I+V)(|q_{ii}| I - Y_i(V))^{-1}$ and
$Y_i(V)=a_i V+\frac12 \sigma_i^2 V^2 +\int_0^\infty (e^{Vx}-I)\nu_i(x)dx$.

For the numerical experiments we considered two problems that we describe below.

\subsection*{Example 1}
The phase generator is the $n\times n$ matrix
\[
Q = \begin{bmatrix}
-\alpha & \alpha \\
 & \ddots & \ddots \\
 & &  \ddots & \alpha \\
\alpha & & & -\alpha  
\end{bmatrix}.
\]
The Brownian motion is defined by $a = \rho\ones$, 
with $\rho < 0$, and $\sigma= \ones$.
There is no jump when there is a change of
phase, hence
$
U(0) = \ones\tr  \cdot \ones$ 
and $\mu(x)=0 $ for all $x >0$.
Jumps during phase $i$, for  $1 \leq i \leq n$,  all have the same density
$\nu(x) = \lambda \bm\tau\tr e^{Tx} (-T \ones)$
where $(\bm\tau, T)$ is a phase-type distribution used in  \cite{blm20}: 
define the $\ell\times \ell$ transition matrix
\[
\widehat T = \begin{bmatrix}
-(c+s) & (1/r_1) & (1/r_1)^2  & \ldots & (1/r_1)^{{\ell}-1}  \\
(r_2/r_1) & -(r_2/r_1) &  & &  \\
(r_2/r_1)^2 &  & -(r_2/r_1)^2 &  &  \\
\vdots   & & & \ddots\\
(r_2/r_1)^{{\ell}-1}  &  &  & & -(r_2/r_1)^{{\ell}-1} 
\end{bmatrix},
\]
where $s = (1/r_1) +(1/r_1)^2 + \cdots + (1/r_1)^{{\ell}-1}$, and the
empty entries are zeros.  The parameters must satisfy the conditions
$r_1 >1$, $r_1>r_2>0$, $c >0$.  The initial probability vector is
$\bm\tau\tr = [1, 0,\cdots, 0]$ and the matrix $T$ is defined as 
\[
T= (-\bm\tau\tr \widehat T^{-1} \ones) \widehat T, 
\]
  such that the expected jump size
is $-\bm\tau\tr T^{-1} \ones = 1$.
The overall drift is $\kappa = \rho + \lambda$.

We take the parameters
$r_1=2$, $r_2=1$,  $c=1.5$ and ${\ell}=10$ like in \cite{blm20}, and
$\lambda=0.1$, $\alpha = 1$, $n =8$, $\rho=-1$,
so that the Brownian drift is constant and equal to $-1$.
With these parameters the value of $\tau^*$ in \eqref{eq:tauc} is
$\tau^*\approx 1.27$  by using the quadratic equation in
\eqref{eq:taus2}, it is $\tau^*\approx 1.33$  by using \eqref{eq:taus3}.

We have compared the convergence speed of the $U$-based, QME-based and NARE algorithms with $X_0=0$ and, for the $U$-based algorithm, with different values of $\tau$. For the QME-based and NARE-based algorithm we have chosen the value $\tau_{opt}=1.33$.
In Figure \ref{fig:exa} we report the  error $\| X_k-X_{k-1}\|_\infty$
as a function of $k$, where $X_k$ is the approximation at step $k$ provided by each algorithm.

 We observe that the QME and NARE-based algorithms have the same
 convergence rate and are fastest. The convergence of the $U$-based algorithm depends on $\tau$ and reaches the fastest speed with the largest value of $\tau$. The actual convergence rates of the QME and NARE-based algorithms, and $U$-based algorithm with $\tau=\tau_{opt}$, coincide with the corresponding theoretical estimates, that is $\rho(R)$ and $\rho(\widehat R)$ of \eqref{eq:R0} and \eqref{eq:Rhat}, respectively, namely 0.10 for the QME and NARE-based algorithms,
 0.35 for the $U$-based algorithm.

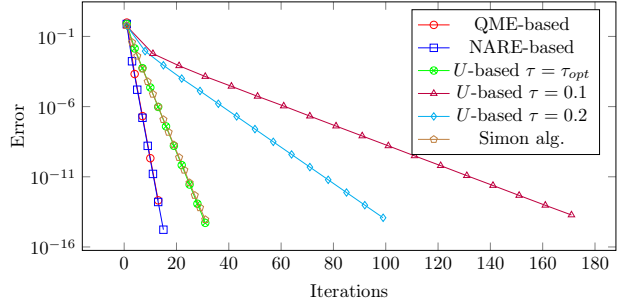
\begin{figure}[ht]
        \begin{center}
        \begin{tikzpicture}[scale =0.7]
        \begin{semilogyaxis}[
                legend pos = north east, width = .9\linewidth, height = .3\textheight,
                xlabel = {Iterations}, 
                ylabel = {Error}, title = {}]
            \addplot[mark=o,red,each nth point={3}]table {uqmeCR_ExA_X00.txt}; 
             \addplot[mark=square,blue,each nth point={2}] table {nare_ExA_X00.txt};
                         \addplot[mark=otimes,green,each nth point={3}] table {uqmeUbased_ExA_X00.txt}; 
             \addplot[mark=triangle,purple,each nth point={10}] table {uqmeUbased_ExA_X00_t01.txt};
              \addplot[mark=diamond,cyan,each nth point={7}] table {uqmeUbased_ExA_X00_t02.txt};
               \addplot[mark=pentagon,brown,each nth point={2}] table {sim_ExA_X00.txt};
            \legend{QME-based, NARE-based,$U$-based $\tau=\tau_{opt}$, $U$-based $\tau=0.1$, $U$-based $\tau=0.2$, Simon alg.}
        \end{semilogyaxis}
        \end{tikzpicture}
        
\end{center}\caption{Example 1. Error per step for different algorithms, with $X_0=0$. In this example $\tau_{opt}= 1.33$.}\label{fig:exa}
\end{figure}

By choosing as starting approximation $X_0=I$ for the QME and $U$-based algorithms, and $X_0=\Delta_b$ for the NARE-algorithm, the convergence is faster, as shown in Figure \ref{fig:exa2}.
As in the case $X_0=0$, the QME and NARE-algorithms have the same convergence rate and faster convergence than the $U$-based algorithm.

\begin{figure}[ht]
        \begin{center}
        \begin{tikzpicture}[scale = 0.7]
        \begin{semilogyaxis}[
                legend pos = north east, width = .9\linewidth, height = .3\textheight,
                xlabel = {Iterations}, 
                ylabel = {Error}, title = {}]
            \addplot[mark=o,red,each nth point={3}]table {uqmeCR_ExA_X0I.txt}; 
             \addplot[mark=square,blue,each nth point={2}] table {nare_ExA_X0Db.txt};
                         \addplot[mark=otimes,green,each nth point={2}] table {uqmeUbased_ExA_X0I.txt}; 
              \addplot[mark=triangle,purple,each nth point={5}] table {uqmeUbased_ExA_X0I_t01.txt};
              \addplot[mark=diamond,cyan,each nth point={5}] table {uqmeUbased_ExA_X0I_t02.txt};
               \addplot[mark=pentagon,brown,each nth point={3}] table {sim_ExA_X0Db.txt};
            \legend{QME-based, NARE-based,$U$-based $\tau=\tau_{opt}$, $U$-based $\tau=0.1$, $U$-based $\tau=0.2$, Simon alg.}
        \end{semilogyaxis}
        \end{tikzpicture}
        
\end{center}\caption{Example 1. Error per step for different algorithms, with $X_0=I$ for the QME-based algorithms, $X_0=\Delta_b$ for the NARE-based algorithm. In this example $\tau_{opt}= 1.33$.}\label{fig:exa2}
\end{figure}

The faster convergence of the sequence generated by setting $X_0=I$, or $X_0=\Delta_b$, is not a surprise in fact the same behaviour is encountered in the standard fixed point iterations for power series matrix equations. We have not performed a theoretical convergence analysis in this case. In our opinion it should be possible to carry out this analysis by following the same arguments and tools as in \cite{blm:book}.

In Table \ref{table:cpu} we report the CPU time required by the
different algorithms, with $X_0=0$, and $X_0=I$ or $X_0=\Delta_b$ (for
the NARE-based and the Simon  algorithms). We observe that QME and
NARE algorithms perform similarly, and are substantially faster than
the $U$-based algorithm. A further acceleration is obtained by
choosing $X_0=I$ and $X_0=\Delta_b$ for the QME and NARE-based
algorithm, respectively. 
In Figures \ref{fig:exa} and \ref{fig:exa2}, and in Table
\ref{table:cpu}, we systematically see that 
the  Simon algorithm and the best $U$-based algorithm have similar
performances.  The Breuer algorithm is significantly slower for $X_0=0$.

\begin{table}[h]
\centering
\begin{tabular}{c|ccccc }
$X_0$ & $U$-based & QME & NARE & Simon alg. & Breuer alg.\\ \hline
$0$ & 1.6 & 0.77 & 0.74 & 1.55 & 12.0 (0.46) \\
$I/\Delta_b$ & 1.15 & {\bf 0.42} & {\bf 0.42} & 1.15 & --
\end{tabular}\caption{Example 1. CPU time (in seconds) for the
  different algorithms, with starting approximation $X_0=0$ and
  $X_0=I$ for $U$-based and QME-based, $X_0=\Delta_b$ for the
  NARE-based and Simon algorithm. For the Breuer algorithm the
  preprocessing time is given in parenthesis.}\label{table:cpu}
\end{table}

In Table \ref{table:res} we compare the residual error for the approximation obtained by using the QME-based algorithm with different values of $\tau$. We observe that smaller values of $\tau$ yield lower accuracy, therefore the value $\tau_{opt}$ is the best choice.

\begin{table}[h]
\centering
\begin{tabular}{c|ccccc }
& $\tau_{opt}$ &  $\tau=10^{-1}$ &  $\tau=10^{-3}$ &  $\tau=10^{-5}$ &  $\tau=10^{-7}$ \\ \hline
Residual & 5.9e-16 & 1.0e-14 & 7.9e-11 & 1.0e-7 &1.0e-2\\
\end{tabular}\caption{Example 1. Residual error for different values of $\tau$ for the QME-based algorithm with $X_0=0$.}\label{table:res}
\end{table}

Finally, we have compared the QME-based algorithm to the Simon algorithm on the basis of CPU time and residual error in the case of Example 1 for increasing values of the matrix size $n$. As it turns out from Table \ref{tab:n}, the QME-based algorithm is faster than Simon algorithm by a facto3 between 2.5 and 3, and provides a residual error which is smaller of one order of magnitude with respect to Simon algorithm.

\begin{table}
\begin{tabular}{|l|l|lllllll|}
\hline
&$n$& 
10 & 20 & 40 & 80 & 160 & 320 & 640\\ \hline
\multirow{3}*{cpu} & QME 
& 0.46 & 0.75 & 2.12 & 4.95 & 18.1 & 69.7 &409.5 \\
& Simon & 
1.28 & 2.04 & 6.28 & 13.6 & 46.3 & 176.5 & 1023.5  \\
& ratio & 
2.8 & 2.7 & 3.0 & 2.7 & 2.6 & 2.5 & 2.5\\ \hline

\multirow{2}*{error} & QME &  
7.0e-16 & 7.0e-16 & 1.0e-15 & 1.1e-15 & 2.4e-15 & 6.5e-15 & 7.6e-15 \\
& Simon & 
3.8e-15 & 8.3e-15 & 1.4e-14 & 3.0e-14 & 5.7e-14 &1.1e-13& 2.6e-13 \\
\hline
\end{tabular}\caption{CPU time in seconds, and residual error, for the QME-based algorithm and the Simon algorithm concerning Example 1 for different values of the matrix size $n$.}\label{tab:n}
\end{table}

\section*{Example  2}

This problem is inspired from \cite{dl14}. The level normally evolves in Phases 1 or 2;
occasionally it is in Phase 3 for short periods of time.  In Phases 1
and 2, the volatility is moderate, the drift is positive in Phase 1,
negative in Phase 2.  There is an exponential jump when the phase
moves from 1 to 2 and back.  In Phase 3, the volatility is high, the
drift is negative, and there are repeated jumps.
The infinitesimal generator is
\[
Q = \begin{bmatrix}
- \alpha - \omega & \alpha & \omega \\
\alpha & -\alpha - \omega & \omega \\
\beta & \beta & -2 \beta
\end{bmatrix}
\]
 with 
$\alpha = 1$, $\omega = 0.25$, $\beta = 0.5$, so that the process
spends 1/5th of the time in Phase 3.
Moreover, 
$a\tr = [-2 , 1 , -\gamma]$, $\sigma\tr= [1 , 1 , 10]$,
where $\gamma > 0$ is a parameter that will allow us to change the dynamics
of the process, and
\[
\mu(x) = \begin{bmatrix}
 \cdot & \eta e^{-\eta x} & \cdot \\
 \eta e^{-\eta x} & \cdot & \cdot \\
 \cdot & \cdot & \cdot 
\end{bmatrix}
\]
where $\eta > 2$,
$\nu_1(x) = \nu_2(x) = 0$ for all $x \geq 0$, and
$\nu_3(x) = \gamma \, e^{- x }$.
The asymptotic drift is
\[
\kappa =\beta(2 \alpha / \eta -1)/(\omega+2 \beta).
\]
We have chosen  $\gamma = 10^{-4}$ and $\eta = 4$. 

The observations, in terms convergence, CPU time and accuracy, are the
same as for  Example 1. We report in Figure \ref{fig:exd} the error as
a function of the number of iterations of the QME, the $U$-based 
algorithms with $X_0=0$ and $X_0=I$, and the Simon algorithm with
$X_0=0$ and $X_0= \Delta_b$. The fastest convergence is obtained by
the QME-based algorithm with $X_0=I$ or $X_0=\Delta_b$.

For this problem also, the actual convergence rates of the QME  and $U$-based algorithm with $\tau=\tau_{opt}$, coincide with the corresponding theoretical estimates, that is $\rho(R)$ and $\rho(\widehat R)$ of \eqref{eq:R0} and \eqref{eq:Rhat}, respectively, namely 0.62 for the QME-based algorithm,
 0.95 for the $U$-based algorithm.

\begin{figure}
        \begin{center}
        \begin{tikzpicture}[scale = 0.7]
        \begin{semilogyaxis}[
                legend pos = north east, width = .9\linewidth, height = .3\textheight,
                xlabel = {Iterations}, 
                ylabel = {Error}, title = {}]
            \addplot[mark=o,red,each nth point={2}]table {uqmeCR_ExD_X0I.txt}; 
             \addplot[mark=square,blue,each nth point={5}] table {uqmeCR_ExD_X00.txt};
                         \addplot[mark=otimes,green,each nth point={3}] table {uqmeUbased_ExD_X0I.txt}; 
              \addplot[mark=triangle,purple,each nth point={20}] table {uqmeUbased_ExD_X00.txt};
              \addplot[mark=pentagon,brown,each nth point={3}] table {sim_ExD_X0Db.txt};
                \addplot[mark=oplus,black,each nth point={20}] table {sim_ExD_X00.txt};
            \legend{QME-based $X_0=I$, QME-based $X_0=0$, $U$-based $X_0=I$, $U$-based $X_0=0$, Simon alg $X_0=\Delta_b$, Simon alg $X_0=0$}
        \end{semilogyaxis}
        \end{tikzpicture}
\end{center}\caption{Example 2. Error per step for different algorithms, with $X_0=I$ for the QME-based algorithms, $X_0=\Delta_b$ for the NARE-based algorithm. In this example $t_{opt}= 0.35$.}\label{fig:exd} 
\end{figure}
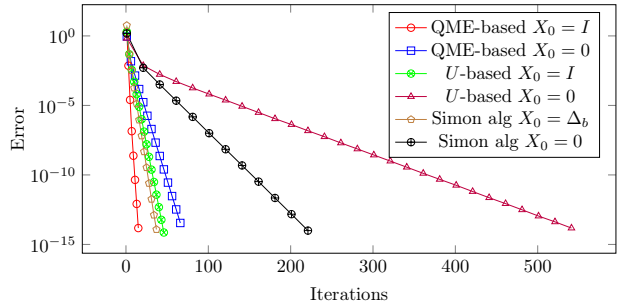

In Table \ref{table:cpu2} we report the CPU time required by the different algorithms, with $X_0=0$, $X_0=I$ and $X_0=\Delta_b$ for the NARE-based algorithm and Simon algorithm. Also in this case  QME and NARE algorithms perform similarly, and are substantially faster than the $U$--based algorithm. A further acceleration is obtained by choosing $X_0=I$ and $X_0=\Delta_b$ for the QME and NARE-based algorithm, respectively. Compared with the Simon \cite{matth17} and with the Breuer algorithm \cite{breue08}, the QME and NARE-based algorithms are faster.

\begin{table}[h]
\centering
\begin{tabular}{c|ccccc }
$X_0$ & $U$-based & QME & NARE & Simon alg. & Breuer alg.\\ \hline
$0$ & 11.6 & 1.4 & 1.5 & 4.4 & 1.7 (0.03) \\
$I/\Delta_b$ & 0.96 & {\bf 0.33} & {\bf 0.33} & 0.78 & --
\end{tabular}\caption{Example 2. CPU time (in seconds) for the different algorithms, with starting approximation $X_0=0$ and $X_0=I ~/~ \Delta_b$. For the Breuer algorithm the preprocessing time is between parenthesis.}\label{table:cpu2}
\end{table}

\bibliographystyle{siamplain}


\begin{thebibliography}{10}
\bibitem{bb16}
{\sc L.~Ballotta and E.~Bonfiglioli}, {\em Multivariate asset models using
  {L\'evy} processes and applications}, The European Journal of Finance, 22
  (2016), pp.~1320--1350.

\bibitem{bp:book}
{\sc A.~Berman and R.~J. Plemmons}, {\em Nonnegative Matrices in the
  Mathematical Sciences}, vol.~9 of Classics in Applied Mathematics, SIAM,
  Philadelphia, PA, 1994.

\bibitem{bertoin96}
{\sc J.~Bertoin}, {\em L\'{e}vy {P}rocesses}, Cambridge {U}niversity {P}ress,
  1996.

\bibitem{bim:book}
{\sc D.~A. Bini, B.~Iannazzo, and B.~Meini}, {\em Numerical Solution of
  Algebraic {R}iccati Equations}, vol.~9 of Fundamentals of Algorithms, SIAM,
  Philadelphia, PA, 2012.

\bibitem{blm:book}
{\sc D.~A. Bini, G.~Latouche, and B.~Meini}, {\em Numerical Methods for
  Structured {M}arkov Chains}, Oxford University Press, New York, 2005.

\bibitem{blm20}
{\sc D.~A. Bini, G.~Latouche, and B.~Meini}, {\em {A family of fast fixed point
  iterations for M/G/1-type Markov chains}}, IMA J. Numer. Anal.,  (2021),
  \url{https://doi.org/10.1093/imanum/drab009}.

\bibitem{breue08}
{\sc L.~Breuer}, {\em First passage times for {Markov} additive processes with
  positive jumps of phase type}, J. Appl. Prob., 45 (2008), pp.~779--799.

\bibitem{ren-can-li}
{\sc C.~Chen, R.-C. Li, and C.~Ma}, {\em Highly accurate doubling algorithm for
  quadratic matrix equation from quasi-birth-and-death process}, Linear Algebra
  Appl., 583 (2019), pp.~1--45.

\bibitem{6cinesi}
{\sc C.-Y. Chiang, E.-W. Chu, C.-H. Guo, T.-M. Huang, W.-W. Lin, and S.-F. Xu},
  {\em Convergence analysis of the doubling algorithm for several nonlinear
  matrix equations in the critical case}, SIAM J. Matrix Anal. Appl., 31
  (2009), pp.~227--247.

\bibitem{dikm10}
{\sc B.~D'Auria, J.~Ivanovs, O.~Kella, and M.~Mandjes}, {\em First passage of a
  {M}arkov additive process and generalized {J}ordan chains}, J. Appl. Probab.,
  47 (2010), pp.~1048--1057.

\bibitem{ds17}
{\sc G.~Deelstra and M.~Simon}, {\em Multivariate {European} option pricing in
  a {Markov}-modulated {L\'{e}vy} framework}, J. Comput. Appl. Math., 317
  (2017), pp.~171--187.

\bibitem{dl14}
{\sc S.~Dendievel and G.~Latouche}, {\em Approximations for time-dependent
  distributions in {M}arkovian fluid models}, Methodol. Comput. Appl. Probab.,
  19 (2017), pp.~285--309.

\bibitem{guo-1999}
{\sc C.-H. Guo}, {\em On the numerical solution of a nonlinear matrix equation
  in {M}arkov chains}, Linear Algebra Appl., 288 (1999), pp.~175--186.

\bibitem{guo-2003}
{\sc C.-H. Guo}, {\em On a quadratic matrix equation associated with an
  {$M$}-matrix}, IMA J. Numer. Anal., 23 (2003), pp.~11--27.

\bibitem{higham:book}
{\sc N.~J. Higham}, {\em Functions of Matrices: Theory and Computation}, SIAM,
  2008.

\bibitem{higham-kim-2000}
{\sc N.~J. Higham and H.-M. Kim}, {\em Numerical analysis of a quadratic matrix
  equation}, IMA J. Numer. Anal., 20 (2000), pp.~499--519.

\bibitem{higham-kim-2001}
{\sc N.~J. Higham and H.-M. Kim}, {\em Solving a quadratric matrix equation by
  {N}ewton's method with exact line searches}, SIAM J. Matrix Anal. Appl., 23
  (2001), pp.~303--316.

\bibitem{jr06}
{\sc A.~Jobert and L.~Rogers}, {\em Option pricing with {Markov}-modulated
  dynamics}, SIAM J. Control Optimization, 44 (2006), pp.~2063--2078.

\bibitem{ln17}
{\sc G.~Latouche and G.~T. Nguyen}, {\em Analysis of fluid flow models},
  Queueing Models and Service Management, 1 (2018), pp.~1--29.

\bibitem{lr99}
{\sc G.~Latouche and V.~Ramaswami}, {\em Introduction to Matrix Analytic
  Methods in Stochastic Modeling}, SIAM, Philadelphia PA, 1999.

\bibitem{lr13}
{\sc S.~Li and J.~Ren}, {\em The maximum severity of ruin in a perturbed risk
  process with {Markovian} arrivals}, Statistics and Probability Letters, 83
  (2013), pp.~993--998.

\bibitem{lc07}
{\sc Y.~Lu and C.~C.-L. Tsai}, {\em The expected discounted penalty at ruin for
  a {Markov}-modulated risk process perturbed by diffusion}, North Amer.
  Actuarial J., 11 (2007), pp.~136--149.

\bibitem{meng}
{\sc J.~Meng, S.-H. Seo, and H.-M. Kim}, {\em Condition numbers and backward
  error of a matrix polynomial equation arising in stochastic models}, J. Sci.
  Comput., 76 (2018), pp.~759--776.

\bibitem{benny-2012}
{\sc J.~F. P\'{e}rez, M.~Telek, and B.~Van~Houdt}, {\em A fast {N}ewton's
  iteration for {$M/G/1$}-type and {$GI/M/1$}-type {M}arkov chains}, Stoch.
  Models, 28 (2012), pp.~557--583.

\bibitem{prabh98}
{\sc N.~U. Prabhu}, {\em Stochastic Storage Processes}, Springer, New York,
  1998.

\bibitem{rhee-2010}
{\sc N.~H. Rhee}, {\em Note on functional iteration technique for {$M/G/1$}
  type {M}arkov chains}, Linear Algebra Appl., 432 (2010), pp.~1042--1048.

\bibitem{sato99}
{\sc K.-I. Sato}, {\em L\'evy Processes and Infinitely Divisible
  Distributions}, Cambridge University Press, Cambridge, UK, 1999.

\bibitem{seokim}
{\sc J.-H. Seo and H.-M. Kim}, {\em Convergence of pure and relaxed {Newton}
  methods for solving a matrix polynomial equation arising in stochastic
  models}, Linear Algebra Appl., 440 (2014), pp.~34--49.

\bibitem{seo2018}
{\sc S.-H. Seo, J.-H. Seo, and H.-M. Kim}, {\em A modified {N}ewton method for
  a matrix polynomial equation arising in stochastic problem}, Electr. J. of
  Linear Algebra, 34 (2018), pp.~500--513.

\bibitem{matth17}
{\sc M.~Simon}, {\em Markov-Modulated Processes: {Brownian} Motions, Option
  Pricing and Epidemics}, PhD thesis, Universit\'{e} libre de Bruxelles, 2017.

\bibitem{vldssl06}
{\sc S.~{Van Lierde}, A.~da~Silva~Soares, and G.~Latouche}, {\em Invariant
  measures for fluid queues}, Stochastic Models, 24 (2008), pp.~133--151.

\bibitem{varga}
{\sc R.~S. Varga}, {\em Matrix Iterative Analysis}, vol.~27 of Springer Series
  in Computational Mathematics, Springer-Verlag, Berlin, expanded~ed., 2000.


\end{thebibliography}

\end{document}